\documentclass[11pt,reqno]{amsart}
\RequirePackage{color}
\usepackage[T1]{fontenc}
\usepackage[utf8]{inputenc}
\usepackage[all]{xy}
\usepackage{graphicx}
\usepackage{lmodern}
\usepackage{amsmath,amssymb,amsthm,mathtools}
\usepackage[margin=1.05in]{geometry}
\usepackage[expansion=false]{microtype}
\usepackage{enumitem}
\usepackage{xcolor}
\usepackage[hidelinks]{hyperref}

\newtheorem{theorem}{Theorem}[section]
\newtheorem{proposition}[theorem]{Proposition}
\newtheorem{lemma}[theorem]{Lemma}

\newtheorem*{maintheorem}{Theorem A}
\newtheorem*{maintheorem2}{Theorem B}
\newtheorem*{maincorollary}{Corollary C}
\newtheorem*{maincorollary2}{Corollary D}
\newtheorem*{maincorollary3}{Corollary E}
\newtheorem*{maincorollary4}{Corollary F}
\newtheorem*{remarkk}{Remark}
\theoremstyle{definition}
\newtheorem{definition}[theorem]{Definition}

\newtheorem{example}[theorem]{Example}
\theoremstyle{remark}
\newtheorem{remark}[theorem]{Remark}

\DeclareMathOperator{\Hom}{Hom}
\DeclareMathOperator{\End}{End}
\DeclareMathOperator{\Ext}{Ext}
\DeclareMathOperator{\Epi}{Epi}
\DeclareMathOperator{\coker}{coker}
\DeclareMathOperator{\im}{im}
\DeclareMathOperator{\rk}{rank}

\DeclareMathOperator{\rep}{rep}
\DeclareMathOperator{\Irr}{Irr}
\DeclareMathOperator{\GL}{GL}
\DeclareMathOperator{\proj}{proj}
\newcommand{\modA}{\operatorname{mod} A}
\DeclareMathOperator*{\modu}{mod}

\DeclareMathOperator*{\Fac}{Fac}

\newcommand{\kk}{k}
\newcommand{\dd}{\mathbf d}
\newcommand{\bb}{\mathbf b}
\newcommand{\D}{D}

\newcommand{\calI}{\mathcal I}
\newcommand{\calD}{\mathcal D}

\numberwithin{equation}{section}
\setlist[enumerate]{label=\textup{(\roman*)},leftmargin=2.2em,itemsep=3pt}
\title[The $2$nd b-BT Conjecture for $E$-infinite algebras]{A Proof of the Second brick-Brauer-Thrall Conjecture for $E$-infinite algebras}
\author[Kaveh Mousavand, Charles Paquette]{Kaveh Mousavand, Charles Paquette} 
\address{Unit of Representation Theory and Algebraic Combinatorics, Okinawa Institute of Science and Technology, Okinawa, Japan}
\email{mousavand.kaveh@gmail.com}
\address{Department of Mathematics and Computer Science, Royal Military College of Canada, Kingston ON, Canada}
\email{charles.paquette.math@gmail.com}

\begin{document}

\begin{abstract}

For any finite-dimensional algebra $A$ over an algebraically closed field, we prove that if $A$ is $E$-infinite, then the Second brick-Brauer--Thrall Conjecture holds for $A$. In fact, we show that if there is a rational ray outside the $\tau$-tilting fan of $A$, then $A$ admits an infinite family of $\theta$-stable bricks of dimension $d$, for some positive integer $d$ and weight $\theta$. 
 To show our main result, from any $\tau$-regular component of $A$ with no dense orbit, we obtain a brick component that also has no dense orbit and whose points in general position are $\theta$-stable modules for a given weight~$\theta$. As an important consequence, we conclude that $A$ is stably-discrete (also called multiplicity-free) if and only if $A$ is $E$-finite. Moreover, we prove that Demonet’s (lattice point) Conjecture~implies the stable Second brick-Brauer--Thrall Conjecture. For $E$-infinite algebras, our results settle several open conjectures. Together with our earlier work  and recent developments, this leads to important reductions in the study of some challenging open problems.
\end{abstract}

\maketitle

\tableofcontents

\section{Introduction and main results}\label{Sec: Introduction}

Throughout, $A$ is a finite-dimensional basic associative unital
$\kk$-algebra, where $\kk$ is algebraically closed. Hence, we can assume $A\simeq kQ/I$, for a finite quiver $Q$ and an admissible ideal $I$.
We let $n$ denote the rank of $A$, which is the same as the number of vertices of $Q$. By $\modu A$ we denote the category of finite-dimensional left $A$-modules. Unless specified otherwise, by a module we always mean an object in $\modu A$, considered up to isomorphism. Moreover, for a module $M$, by $\underline{\dim}M$ we denote the dimension vector of $M$, an integer vector belonging to $\mathbb{Z}^n_{\geq 0}$. Let $\D=\Hom_\kk(-,\kk)$ be the standard duality, and by $\tau$ we denote the Auslander--Reiten translation in $\modu A$. 

\medskip

Recall that a nonzero module $B$ is a \emph{brick} provided $\End_A(B)$ is a division $\kk$-algebra. Since $\kk$ is algebraically closed, $B$ in $\modu A$ is a brick if and only if $\End_A(B)\cong\kk$. Henceforth, we identify $K_0(\proj A)$ with the lattice of integral weights (also known as $g$-vectors) on $K_0(\modA)$ using the standard pairing $\langle[P],\underline{\dim}M\rangle=\dim\Hom_A(P,M)$. Consequently, for a weight $\theta \in K_0(\proj A)$, by $\theta(M)$ we denote the value of $\theta$ on $\underline{\dim}M$.
Following \cite{Ki}, a nonzero module $B$ is said to be \emph{$\theta$-stable} if $\theta(B)=0$ and $\theta(L)<0$ for every
nonzero proper submodule $L\subsetneq B$. Replacing the above strict inequality by $\theta(L)\leq0$ yields \emph{$\theta$-semistability}. Stable modules are necessarily bricks (but the converse is not necessarily true), and it is known that stability is an open
condition in a fixed representation variety; see \cite{Ki} for more details.

\medskip

For a dimension vector $\dd=(d_1,\ldots,d_n)$ in $\mathbb{Z}^n_{\geq 0}$, let $\rep(A,\dd)$ be
the variety that parametrizes representations of $A$ of dimension vector $\dd$.
The group $G_\dd=\prod_{i=1}^n\GL_{d_i}(\kk)$ acts on this affine space via conjugation, so that its orbits correspond to isomorphism classes of modules in $\modu A$. 
Let $\Irr(A,\dd)$ denote the set of irreducible components of $\rep(A,\dd)$. Moreover, $\Irr(A)$ denotes the union of all $\Irr(A,\dd)$, where $\dd$ runs through $\mathbb{Z}^n_{\geq 0}$. For $Z\in \Irr(A,\dd)$, define
\[
 c(Z):=\dim Z-\max_{X\in Z}\dim(G_\dd\cdot X),
 \qquad \text{and} \qquad
 h(Z):=\operatorname{min}_{X\in Z}\dim\Hom_A(X,\tau X),
\]
and we note that these two values are attained on some dense open subset of $Z$, hence they are called generic values.  The component $Z$ is said to be \emph{generically $\tau$-regular} (or
\emph{$\tau$-regular}, for short), if $c(Z)=h(Z)$.  We remark that $\tau$-regular components have also appeared under different names in the literature, but we adopt this terminology for consistency with the more recent studies on these components; for more details, see \cite{BS} and the references therein. Finally, a component $Z \in \Irr(A)$ is a \emph{brick component} if $Z$ contains a brick; equivalently, $Z$ contains a non-empty open subset consisting of bricks.

\medskip

To state our main results, we first recall the following definition; for more details, see \cite{MP2}.

\begin{definition}\label{Main definition}
Let $A$ be an algebra. Then,
\begin{enumerate}
    \item $A$ is \emph{brick-finite} if there are only finitely many (isomorphism classes of) bricks in $\modu A$.
    \item $A$ is \emph{brick-continuous} if, for some positive integer $d$, there are infinitely many (pairwise non-isomorphic) bricks of dimension $d$.
    \item $A$ is \emph{$E$-finite} if every $\tau$-regular
component $Z \in \Irr(A)$ satisfies $c(Z)=0$.
\end{enumerate}
\end{definition}

For algebra $A$, we say that $A$ is brick-infinite (respectively, $E$-infinite) if $A$ is not brick-finite (respectively, not $E$-finite). Every E-infinite algebra is known to be brick-infinite \cite{AI}, and each brick-continuous algebra is evidently brick-infinite. Notably, the converse implications are still open; see \cite{MP2, MP3} and references therein. 
Observe that the definition of $E$-finiteness above is a convenient geometric formulation of the original definition; see \cite[Definition~7.5]{MP2}. Equivalently, $A$ is $E$-finite if the $\tau$-tilting fan of $A$ contains every rational vector in $K_0(\proj A)_\mathbb R$; see \cite{AI, De}. An element of $K_0(\proj A)$ that lies in the $\tau$-tilting fan of $A$ is called a \emph{rigid} $g$-vector.

\medskip

To motivate our results and put them into perspective, observe that, for an algebra $A$, being brick-continuous is an algebraic property of $A$, while being $E$-infinite can be seen as a geometric property of $A$. In the following, we study the implications between these conceptual algebraic and geometric properties. 
In particular, note that if $M$ is $\tau$-rigid (namely, if $\Hom_A(M,\tau M)=0$), then the closure of the $G_\dd$-orbit of $M$ is a $\tau$-regular component $Z\in \Irr(A)$ with $c(Z)=0$. Moreover, every $\tau$-regular component $Z\in \Irr(A)$ with $c(Z)=0$ is of this form. If $A$ is brick-finite, for each $\tau$-regular component $Z\in \Irr(A)$, we have $c(Z)=0$, and hence $A$ is $E$-finite. To the best of our knowledge, it remained unknown whether for an arbitrary algebra $A$, being $E$-infinite implies being brick-continuous; for more details, see \cite{MP1, MP2, MP3} and references therein.

\medskip

Below, for dimension vectors $\bb$ and $\dd$ in $\mathbb{Z}^n_{\geq 0}$, the order $\bb\leq\dd$ is considered componentwise.

\begin{maintheorem}\label{Thm: A}
Let $A$ be an $E$-infinite algebra. Then, there exist a dimension vector $\bb$ and an integral weight $\theta$ such that $A$ admits infinitely many $\theta$-stable bricks of dimension vector $\bb$. 
In particular, every $E$-infinite algebra is brick-continuous.
\end{maintheorem}

Before we discuss some important consequences of the above theorem, let us remark that in Section \ref{Sec:proofs} we prove a stronger result on $E$-infinite algebras.  
In fact, for each $\tau$-regular component $Z \in \Irr(A,\dd)$ with $c(Z)>0$, we obtain a brick component $W\subseteq\rep(A,\bb)$, with $\mathbf{0}\neq\bb\leq\dd$, such that $c(W)>0$ and the general elements of $W$ are $\theta$-stable, for an explicit integral weight $\theta$; see Theorem \ref{Thm: from tau-regular comp. to brick comp}.

\medskip
Recall that an open problem on $g$-vectors, often referred to as \emph{Demonet's conjecture}, asserts that every brick-infinite algebra is $E$-infinite. This is an equivalent reformulation of a question posed by Demonet \cite[Question 3.49]{De} on vectors outside the $\tau$-tilting fan; see \cite[Conjecture~1.2]{MP1}. 
We also recall that the \emph{Stable Second brick-Brauer--Thrall conjecture} states that every brick-infinite algebra admits infinitely many (pairwise non-isomorphic) bricks of a fixed dimension vector that are stable for some common integral weight. This formulation appeared in \cite[Conjecture~5.2]{Pf}, which is a variation of \cite[Conjecture 5.4]{CKW} combined with \cite[Conjecture 6.0.1]{Mo1}, the latter nowadays is often called the \emph{Second brick-Brauer--Thrall (2nd bBT) conjecture}. More specifically, observe that the stable 2nd bBT strengthens the assertion of 2nd bBT by requiring the family of bricks to be stable for some integral weight. 
The above conjectures belong to a longer list of challenging open problems, collectively known as ``brick-Brauer--Thrall (bBT) Conjectures''; for more details and recent developments, see \cite{MP2, MP3, B-L} and references therein. 
As discussed below, our main results contribute to the study of several bBT Conjectures. 

\medskip

We now give a novel characterization of the multiplicity-free algebras introduced in \cite{CKW}, later called the stably-discrete algebras \cite{MP2}.
Recall that $A$ is \emph{multiplicity-free} if for any irreducible component $Z \in \Irr(A)$, the ring of semi-invariant polynomial functions on $Z$ is multiplicity-free; that is, each of its weight spaces is of dimension at most one. For an algebra $A$, being multiplicity-free is equivalent to the property that for each $\theta \in K_0(\proj A)$ and each irreducible component $Z \in \Irr(A)$, the moduli space $\mathcal{M}(A,Z)^{\theta-ss}$ of $S$-equivalences of $\theta$-semistable points is zero dimensional. The latter property justifies the terminology ``stably-discrete''. The reader is referred to \cite{Ki} for the definition of moduli spaces, and to \cite{CKW} for some basic properties of stably-discrete algebras. 
For a treatment of these algebras in the context of several open conjectures and related notions, see \cite[Section 6, 7]{MP2}.
\begin{maintheorem2}
    The algebra $A$ is stably-discrete (that is, multiplicity-free) if and only if it is $E$-finite.
\end{maintheorem2}

To state our next result, we recall that a (possibly infinite) set $\mathcal{S}$ of bricks is said to be a \emph{semibrick} if $\Hom_A(X,Y)=0=\Hom_A(Y,X)$, for any two distinct $X$ and $Y$ in $\mathcal{S}$; see \cite{As1}.
It is worth noting that non-isomorphic $\theta$-stable modules are necessarily Hom-orthogonal, because they are non-isomorphic simple objects in the category of $\theta$-semistable modules. 
 Hence, Theorem A immediately implies the following corollary and generalizes some recent results.

\begin{maincorollary}
If there is a rational ray outside the $\tau$-tilting fan of $A$, then there exists an infinite semibrick $\mathcal{S}$ in $\modu A$ such that every $X \in \mathcal{S}$ is of dimension $d$ and $\theta$-stable, for a fixed positive integer $d$ and stability condition $\theta \in K_0(\proj A)$. 
In particular, Demonet's (lattice point) conjecture implies the Stable Second brick-Brauer--Thrall conjecture.
\end{maincorollary}

We note that the final assertion of the above corollary generalizes some results from \cite{Pf} and \cite{MP1}, recently obtained for a special family of algebras and via different methods. More specifically, the implication was known only for $E$-tame algebras, but our new result holds in general.

\medskip

Our next result relates to another notion of tameness inspired by the property of the $\tau$-tilting fans; see \cite{PY}. In particular, recall that $A$ is \emph{$g$-tame} if the $\tau$-tilting fan of $A$ is dense in $K_0(\proj A)_\mathbb{R}$. The following corollary follows from the previous results, particularly because $E$-finiteness implies $g$-tameness. This improves some earlier studies in this direction; see \cite{MP2} and references therein.

\begin{maincorollary2}
If $A$ is not $g$-tame, the Stable Second brick-Brauer-Thrall conjecture holds for~$A$. 
\end{maincorollary2}

As another corollary of our main theorem, we obtain a new result on an open problem, known as the \emph{rigid-brick conjecture}. As implied by \cite[Conjecture 6.0.1]{Mo1}, it is expected that if almost all bricks in $\modu A$ are rigid (i.e, if $\Ext^{1}_A(B,B)=0$, for all but possibly finitely many bricks $B$), then $A$ is brick-finite. This problem remains open in full generality, but we settle a weaker version of that in this work; for more details and some related results, see \cite{Mo2, MP1, MP4}.

\begin{maincorollary3}
If all but possibly finitely many bricks in $\modu A$ are rigid, then $A$ is $E$-finite. 
\end{maincorollary3}

\medskip

Thanks to the above results and our previous studies \cite{MP5}, for the purpose of proving the Second brick-Brauer-Thrall (2nd bBT) conjecture, we can further reduce the setting of the main problem. 
To state the reduction, first recall that an algebra $A$ is said to be \emph{minimal brick-infinite} if $A$ is brick-infinite but every proper quotient of $A$ by an ideal is brick-finite. Obviously, every brick-infinite algebra admits a minimal brick-infinite quotient.

\begin{maincorollary4}[Reduction]\label{Cor: reduction}
To verify the Second brick-Brauer--Thrall conjecture in full generality (that is, to show that every brick-infinite algebra is brick-continuous), it suffices to settle the statement for algebras $A$ that have the following properties:
\begin{enumerate}
    \item $A$ is minimal brick-infinite.
    \item $A$ is $g$-tame.
    \item Almost all bricks of $A$ are faithful.
\end{enumerate}
\end{maincorollary4}

In the preceding corollary, reduction to (i) follows from a standard argument, reduction to (ii) follows from Corollary D, and reduction to (iii) follows from our earlier work \cite[Theorem 1.4]{MP5}. 

\medskip 

\begin{remarkk}
If $A$ satisfies the three conditions of Corollary F, to verify the 2nd bBT conjecture for $A$, it suffices to construct an integral $g$-vector lying outside of its $\tau$-tilting fan. As shown in \cite[Theorem 1.3]{MP1}, one approach to construct such a $g$-vector is to find an infinite family of indecomposable $\tau$-rigid modules $\{X_i\}_{i\in \mathbb{N}}$ for which $\dim_k \End_A(X_i)$ grows subquadratically; this property is called the \emph{$\tau$-convergence} of $A$. Hence, to fully settle the 2nd bBT conjecture, it suffices to show that each algebra that satisfies all three properties of Corollary~F has the $\tau$-convergence property. Notably, the two properties (i) and (ii), together, do not imply the $\tau$-convergence property; see Example \ref{Ex: non-tau-convergence}.
In light of the above reduction and our previous results (see \cite[Corollary 1.4]{MP1} and \cite[Corollary 8.9]{MP2}), it is interesting to study which algebra $A$ in Corollary~F that satisfies all three properties in fact has the $\tau$-convergence property.
\end{remarkk}

We conclude this section by noting that, beyond their role in our main results, some of the tools developed in this work can provide fresh impetus for a systematic study of stability conditions. More explicitly, Section \ref{sec:presentations} is mainly aimed at finding explicit stability conditions for certain bricks. To put this into context, recall that for a hereditary algebra $A$ and each brick $B$ in $\modu A$, the dimension vector of $B$, say $\mathbf{d}$, is a Schur root. A theorem of Schofield \cite[Theorem 6.1]{Sc} implies that the bricks in general position of dimension vector $\mathbf{d}$ are stable under the stability condition $\langle {\mathbf d}, - \rangle - \langle -, {\mathbf d}\rangle$. 
Our Proposition \ref{prop:stableimage} and Remark \ref{Rem: labeling bricks for distinct supports} are in a similar spirit to that classical result and give explicit stability conditions for the bricks under consideration. To our knowledge, for an arbitrary algebra $A$ and a given stable module $B$ (with possibly unknown stability condition), it is nontrivial how to construct an explicit stability condition for $B$; and remains of great interest.

\section{Preparatory materials and observations}
\label{sec:presentations}

We first record precisely the standard presentation description of the (generically) $\tau$-regular components that will be used in the following. This is the generic-presentation parametrization due to Plamondon
\cite[Theorem~1.2]{Pl}.

\begin{proposition}
\label{prop:presentations}
Let $Z\in \Irr(A,\dd)$ be a $\tau$-regular component.
There are projective modules $P_0, P_1$ and a nonempty open subset $U\subseteq \Hom_A(P_1,P_0)$
such that, for every $p\in U$, the sequence $
 P_1\xrightarrow{p}P_0\to X_p\to 0$ is a minimal projective presentation with $X_p=\coker p$ belonging to $Z$, and $\dim\Hom_A(X_p,\tau X_p)=c(Z)$.
\end{proposition}

We refer to \cite{BS} for more details on the above proposition, the new terminology proposed by the authors and adopted in our work, and proof of a more refined version of the statement.

\medskip

Following the notation of Proposition \ref{prop:presentations}, let $H:=\Hom_A(P_1,P_0)$, and for any $p\in H$ and $M \in \modu A$, consider
$p_M:\Hom_A(P_0,M)\to \Hom_A(P_1,M)$, given by $p_M(q)=qp$. 
That being the case, from a classical result due to Auslander and Reiten \cite{AR}, when $p$ is a minimal projective presentation, we have an exact sequence
\[
 0\to \Hom_A(X_p,M)
 \to \Hom_A(P_0,M)
 \xrightarrow{p_M}\Hom_A(P_1,M) \to \D\Hom_A(M,\tau X_p) \to 0.\]
Using this, for
$\sigma:=[P_0]-[P_1]$ we obtain the weight formula
\begin{equation}
\label{eq:weightformula}
 \sigma(M)=\dim\Hom_A(X_p,M)-\dim\Hom_A(M,\tau X_p)
\end{equation}
for every $p \in U$. 

\medskip

The following lemma will be important in the following. 

\begin{lemma}
\label{lem:incidence}
Let $P_1,P_0$ be projective modules, $H=\Hom_A(P_1,P_0)$, and $M$ be a nonzero module. With the same notation as before, define
\[
 \calD_M=\left\{p\in H\ \middle|\
 \begin{gathered}
 \text{there exists an epi }X_p\twoheadrightarrow M,\;\;
 \Hom_A(M,\tau X_p)\ne 0
 \end{gathered}
 \right\}.
\]
Then, $\calD_M$ is not Zariski-dense.
\end{lemma}

\begin{proof}
Let $N=\dim H$, and set $a:=\dim\Hom_A(P_0,M)$, and $b:=\dim\Hom_A(P_1,M)$.

First, we note that the set of all epimorphisms $f:P_0\to M$ in $\modu A$, denoted by $\Epi_A(P_0,M)$, forms an open subset of $\Hom_A(P_0,M)$, because surjectivity defines maximal-rank conditions on the underlying linear map. 
If this set is empty, then $\calD_M$ is empty and we may take $F_M=\varnothing$.

Henceforth, we assume that $\Epi_A(P_0,M)$ is nonempty. That being the case, consider the variety
\[
 \calI_M=\{(p,q)\in H\times\Epi_A(P_0,M) \,|\, qp=0\}.
\]
For each $q\in\Epi_A(P_0,M)$, consider the linear map $\phi_q: H\longrightarrow\Hom_A(P_1,M)$ given by $\phi_q(p):=qp$. Note that $\phi_q$ is surjective, because $P_1$ is projective. Moreover, observe that the dimension of ${\rm ker}(\phi_q)$ is constant for $q\in \Epi_A(P_0,M)$ and equals $N-b$. Therefore, we can consider $\calI_M$ as a vector bundle of rank $N-b$ over
the irreducible variety $\Epi_A(P_0,M)$, which has dimension $a$.
 Now, we claim that $\calI_M$ is irreducible. To show this, assume otherwise and let $U_1, U_2$ be two nonempty open sets in $\calI_M$ that do not intersect. Since $k$ is algebraically closed, the projection ${\rm pr}:\calI_M \to \Epi_A(P_0,M)$ is an open map, and we get two nonempty open sets ${\rm pr}(U_1)$ and ${\rm pr}(U_2)$. Since $\Epi_A(P_0,M)$ is an irreducible variety, ${\rm pr}(U_1) \cap {\rm pr}(U_2) \neq \emptyset$, and we can take $q$ a point in this intersection. The fiber ${\rm pr}^{-1}(q)$ is the affine space ${\rm ker}(\phi_q)$, hence irreducible. However, both ${\rm pr}^{-1}(q)\cap U_1$ and ${\rm pr}^{-1}(q)\cap U_2$ are non-empty open sets of this affine space and they do not intersect, a contradiction. This proves the irreducibility of $\calI_M$.
Moreover, because $\calI_M$ is a vector bundle of rank $N-b$ over $\Epi_A(P_0,M)$, and $\dim \Epi_A(P_0,M)=a$, we have $ \dim\calI_M=N+a-b$.

\medskip
Now, consider $\pi:\calI_M\to H$ as the canonical projection onto the first factor. From the definitions, it follows that $\mathrm{Im}(\pi)$ consists of those $p:P_1\to P_0$ for which $X_p=\coker p$ admits an epimorphism onto $M$. 
Furthermore, the set $\calD_M$ is a subset of the set of points $p$ in $\mathrm{Im}(\pi)$ for which the rank of $p_M$ is strictly less than $b$. Note that the condition ``the rank of $p_M$ is strictly less than $b$'' forms a closed subset $S$ of $H$. If this subset is proper, then we set $F_M = \overline{\pi(\calI_M)} \cap S$ and we are done. Therefore, we may assume that a point $p$ in general position in $H$ is such that $p_M$ has rank less than $b$. 
That being the case, we consider two possibilities, and we show that the second possibility never occurs.
\begin{enumerate}
    \item If $\mathrm{Im}(\pi)$ is not dense in $H$, then we take
\[
 F_M=\overline{\pi(\calI_M)}\subsetneq H,
\]
which is a proper closed subset containing $\calD_M$, hence we are done.
    \item If $\mathrm{Im}(\pi)$ is dense in $H$, observe that $\pi:\calI_M\to H$ satisfies the assumptions of \cite[Theorem 3, I.8]{Mu}, where we also know that $\calI_M$ is irreducible with $\dim\calI_M=N+a-b$, and $\dim H=N$. Hence, by the aforementioned theorem, we have the generic equality in the dimension of fibers of $\pi$. More precisely, there exists a nonempty open subset $V\subseteq H$ such that every $p\in V$
has a nonempty fiber of dimension
\begin{equation*}
 \dim\pi^{-1}(p)=a-b.
\end{equation*}
On the other hand, for every $p\in\mathrm{Im}(\pi)$,  the fiber is identified with $\pi^{-1}(p)=\Epi_A(X_p,M)$, which is a nonempty open subset of $\Hom_A(X_p,M)=\ker p_M$.
Therefore, we have
\begin{equation*}
 \dim\pi^{-1}(p)=\dim\ker p_M=a-\rk p_M,
\end{equation*}
which implies $\rk p_M=b$, for general $p\in H$. This contradicts our assumption that a point $p$ in general position in $H$ is such that $p_M$ has rank less than $b$.
\end{enumerate}
\end{proof}

The following lemma is a useful tool for constructing bricks. Although it should be well known, we present a short proof for completeness.

\begin{lemma}
\label{lem:brickimage}
For $M$ and $N$ in $\modu A$, let $f\in\Hom_A(M,N)$ be a nonzero map whose image is of minimal dimension. Then $B=\im f$ is a brick.
\end{lemma}

\begin{proof}
For $B=\im f$, consider the factorization of $f$ as $M\xrightarrow{q}B\xrightarrow{\iota}N$, where $q$ is surjective and $\iota$ is injective. For any nonzero
$g\in\End_A(B)$, the composition $\iota g q$ is again nonzero. Hence, by the minimality assumption, we obtain
$\dim B\leq\dim\im(\iota g q)=\dim\im g\leq\dim B$. 
Therefore, $g$ is an isomorphism, which implies the desired result.
\end{proof}

The next result strengthens the preceding construction of bricks via minimal-image in the case where the co-domain is the Auslander--Reiten translate of the domain. We also note that this proposition can be viewed as a stronger version of a similar construction of the stability condition given in \cite[Lemma 2.5]{CKW} for the homogeneous bricks, which was based on some earlier results of Domokos \cite{Do}. Before we state the following proposition, recall that, for each algebra $A$ and every $M$ in $\modu A$, the weight
$\ell:=[A]$ satisfies $\ell(M)=\dim_k M$.

\begin{proposition}\label{prop:stableimage}
Let $X \in \modu A$ with $\Hom_A(X,\tau X)\ne0$, with a minimal projective presentation $P_1\to P_0\to X\to0$. Take a nonzero map $f:X\to\tau X$ with $\dim\im f$ minimal, and put $B=\im f$. 
If $\sigma:=[P_0]-[P_1]$ and $\ell:=[A]$, then brick $B$ is stable for the integral weight
\[
 \vartheta_B=(\dim B)\sigma-\sigma(B)\ell.
\]
\end{proposition}

\begin{proof}
Consider the factorization of $f$ as $X\xrightarrow{q}B\xrightarrow{\iota}\tau X$, where $q$ is surjective and $\iota$ is injective. Moreover, let $ 0\to L\to B\to Q\to 0$ be an exact sequence with $L$ and $Q$ nonzero. 

First, we claim that $\Hom_A(X,L)=0$ and $\Hom_A(Q,\tau X)=0$. Observe that a nonzero map $X\to L$, followed by $L\hookrightarrow B\hookrightarrow\tau X$, would give a nonzero map $X\to\tau X$ with image of dimension at most $\dim L<\dim B$.
Similarly, a nonzero map $Q\to\tau X$, preceded by $X\twoheadrightarrow B\twoheadrightarrow Q$, would have image of dimension at most $\dim Q<\dim B$. Both cases contradict minimality, and this verifies the claim.

Evidently, the inclusion $L\hookrightarrow B\hookrightarrow\tau X$, as well as the surjection $X\twoheadrightarrow B\twoheadrightarrow Q$, are nonzero.
Thus, the weight formula \eqref{eq:weightformula}, together with $\Hom_A(X,L)=0=\Hom_A(Q,\tau X)$, implies that
\[\sigma(L)=-\dim\Hom_A(L,\tau X)<0, \qquad \text{and}
 \qquad \sigma(Q)=\dim\Hom_A(X,Q)>0.\]

Finally, to verify the stability conditions, first note that $\vartheta_B(B)=0$ follows from the definition. 
Moreover, from the additivity of $\sigma$, together with the additivity of dimension, we have
\begin{align*}
 \vartheta_B(L)
 &=(\dim B)\sigma(L)-\sigma(B)\dim L\\
 &=(\dim Q)\sigma(L) + (\dim L)\sigma(L)-\sigma(L)\dim L-\sigma(Q)\dim L\\
 &=(\dim Q)\sigma(L)-(\dim L)\sigma(Q)<0.
\end{align*}
This shows that $B$ is $\vartheta_B$-stable and finishes the proof.
\end{proof}

\begin{remark}
In the above construction (and adapting Lemma \ref{lem:brickimage}), one can use the minimality condition with respect to the dimension vectors, instead of the total dimensions. Namely, we could take $0\ne f \in \Hom_A(X, \tau X)$ such that $\underline{\dim}\,{\rm Im}(f)$ is minimal with respect to componentwise ordering on $K_0(\modu A)\simeq (\mathbb{Z}_{\geq 0})^n$.
\end{remark}

Recall that an inclusion $\mathcal{T}_1 \subsetneq \mathcal{T}_2$ of torsion classes in $\modu A$ is said to be a \emph{cover} relation, and it is denoted by $\mathcal{T}_1 \lessdot \mathcal{T}_2$, provided that for every torsion class $\mathcal{T}'$ in $\modu A$, we have $\mathcal{T}_1 \subseteq \mathcal{T}' \subseteq \mathcal{T}_2$ implies $\mathcal{T}_1= \mathcal{T}'$ or $\mathcal{T}_2= \mathcal{T}'$. 
As shown in \cite{BCZ}, every cover relation of torsion classes is uniquely labeled by a brick in $\modu A$, and each brick in $\modu A$ appears (at least once) as the label of a cover relation of torsion classes. 
Moreover, for $\mathcal{T}_1 \lessdot \mathcal{T}_2$, \cite[Theorem 3.1]{DIJ} implies that $\mathcal{T}_1$ is functorially finite if and only if $\mathcal{T}_2$ is functorially finite. For more details, see the aforementioned papers and references therein.

\medskip

In the following, we further analyze some ideas from Proposition \ref{prop:stableimage} to treat the bricks that label the cover relations between functorially finite torsion classes. First, we introduce a useful terminology. 
Meanwhile recall that, for a simple module $S$, by $P(S)$ and $I(S)$ we denote the indecomposables which are, respectively, the projective cover and injective envelope of $S$.

\begin{definition}
For a torsion class $\mathcal{T}$ in $\modu A$, the \emph{support} of $\mathcal{T}$ is the set of simple modules $S$ in $\modu A$ for which $\Hom_A(P(S),\mathcal{T})\neq 0$; that is, for some $M\in \mathcal{T}$, we have $\Hom_A(P(S),M)\neq 0$.
\end{definition} 

To further clarify the notion of support of a torsion class, recall that for each torsion pair $(\mathcal{T},\mathcal{F})$ in $\modu A$, and every simple module $S$, we have $S \in \mathcal{T}$ or $S \in \mathcal{F}$; and evidently not both. 
If $S \in \mathcal{T}$, then $S$ obviously belongs to the support of $\mathcal{T}$, but the converse is generally not true. 

\medskip

Our proof of the following lemma relies on the support of torsion classes. 
Note that if $B$ is the brick which labels the cover relation $\mathcal{T}_1 \lessdot \mathcal{T}_2$ of torsion classes, then $B \in \mathcal{T}_2\setminus \mathcal{T}_1$; see \cite{BCZ}.

\begin{lemma}\label{Lem: cases for labeling bricks}
Let $(\mathcal{T}_1,\mathcal{F}_1)$ and $(\mathcal{T}_2,\mathcal{F}_2)$ be torsion pairs in $\modu A$. 
If $\mathcal{T}_1 \lessdot \mathcal{T}_2$ is a cover relation of functorially finite torsion classes, then its labeling brick $B$ satisfies at least one of the following:

\begin{enumerate}
    \item $B$ is a minimal image of $f: X \to \tau X$, for some $X= T \oplus \tau^{-1}F$, with $F\in \mathcal{F}_1$ and $T\in \mathcal{T}_2$.
    \item $B$ is a submodule of an indecomposable injective module.
\end{enumerate}
\end{lemma}

Before we show the above statement, note that under the hypothesis of Lemma \ref{Lem: cases for labeling bricks}, the labeling brick $B$ is known to be the quotient of an indecomposable $\tau$-rigid module $T \in \mathcal{T}_2$; see \cite{BCZ, DIJ}. 
In the spirit of Proposition \ref{prop:stableimage}, we give a realization of (almost all) such labeling bricks as minimal images of nonzero morphism $f$ in $\Hom_A(X,\tau X)\ne0$, for some $X \in \modu A$ constructed below. 

\begin{proof}[Proof of Lemma \ref{Lem: cases for labeling bricks}]
First, we note that $B$ is a minimal extending brick for $\mathcal{T}_1$ in the sense of \cite[Definition 1.1]{BCZ}, also called a torsion-free almost torsion module with respect to $(\mathcal{T}_1, \mathcal{F}_1)$. In particular, there exists an indecomposable $\tau$-rigid-module $T$ in $\mathcal{T}_2$ with an epimorphism $u: T \to B$, and each non-zero morphism $T \to B$ is an epimorphism; see \cite[Proposition 2.6]{BCZ}. Dually, there is an indecomposable $\tau^{-1}$-rigid module $F$ in $\mathcal{F}_1$  with a monomorphism $v: B \to F$, and each non-zero morphism $B \to F$ is a monomorphism. Because $T$ is $\tau$-rigid, we have $\Hom_A(B, \tau T)=0$, and, since $F$ is $\tau^{-1}$-rigid, we get $\Hom_A(\tau^{-1}F,B)=0$.

\medskip

Now we consider two mutually exclusive situations with respect to the supports of $\mathcal{T}_1$ and $\mathcal{T}_2$.

\emph{Let $\mathcal{T}_1$ and $\mathcal{T}_2$ have the same support}: In this case, the indecomposable module $F$ in $\mathcal{F}_1$ cannot be injective; otherwise, $F=I(S)$, for a simple module $S$, which implies $S$ lies outside the support of $\mathcal{T}_1$ but in the support of $\mathcal{T}_2$, a contradiction. 
Set $X := T \oplus \tau^{-1}F$ and note that $X\in \mathcal{T}_2$. Evidently, we get a surjection $X \to B$, and an injection $B \to \tau X$, hence a nonzero map $f: X \to \tau X$ given by the following composition
$$X=T \oplus \tau^{-1}F \stackrel{\left[\begin{array}{cc}
       u  & 0 \\
    \end{array}\right]}{\longrightarrow} B \stackrel{\left[\begin{array}{c}
        0 \\
        v
    \end{array}\right]}{\longrightarrow} \tau T \oplus F = \tau X.$$ 
    Since every nonzero morphism $T \to B$ is an epimorphism and $\Hom_A(\tau^{-1}F,B)=0$, the map $f$ has minimal (nonzero) image, which is $B$, as claimed in (i).

\medskip

\emph{Let $\mathcal{T}_1$ and $\mathcal{T}_2$ have different supports}: In this case, the indecomposable module $F$ in $\mathcal{F}_1$ can be injective. 
If $F$ is not injective, then we apply the same argument as above, and conclude (i). Otherwise, $F$ is injective and $B$ is a submodule of $F$, hence (ii) holds. 
\end{proof}

From the previous lemma and some of our earlier results, we obtain the following proposition. For a torsion class $\mathcal{T}$ in $\modu A$, let $\mathcal{F}$ denote the corresponding torsion-free class $\mathcal{F}:=\mathcal{T}^{\perp}$.

\begin{proposition}\label{Prop: min-image characterization of left-finite bricks}
For any algebra $A$, all but finitely many bricks $B$ that label a cover relation $\mathcal{T}_1 \lessdot \mathcal{T}_2$ of functorially finite torsion classes in $\modu A$ are minimal images of non-zero morphisms $f_B:X_B \to \tau X_B$, for some $X_B$ in $\mathcal{T}_2$.
More explicitly, we have $X_B=T_B \oplus \tau^{-1}F_B$, for a unique indecomposable $\tau$-rigid module $T_B \in \mathcal{T}_2$ and a unique indecomposable $\tau^{-1}$-rigid module $F_B \in \mathcal{F}_1$.
\end{proposition}
\begin{proof}
Let $B$ be the brick that labels the cover relation $\mathcal{T}_1 \lessdot \mathcal{T}_2$. 
By Lemma \ref{Lem: cases for labeling bricks}, we only need to argue that only finitely many such bricks in $\modu A$ appear as a submodule of an indecomposable injective modules. 
As shown in \cite[Theorem 6.1]{MP5}, the labeling brick $B$ has open orbit. Hence, for positive integer $d$, there are only finitely many such bricks of dimension $d$. Obviously, there is a global bound on the dimension of the indecomposable injectives in $\modu A$. Therefore, only finitely many such labeling bricks fall in case (ii) of Lemma \ref{Lem: cases for labeling bricks}. This implies the desired result. 
Moreover, $T_B \in \mathcal{T}_2$ and $F_B \in \mathcal{F}_1$ are given in the proof of Lemma \ref{Lem: cases for labeling bricks}.
\end{proof}

The above proposition is of particular interest when $A$ is a brick-infinite algebra. This is the case if and only if there exist infinitely many functorially finite torsion classes in $\modu A$, if and only if infinitely many non-isomorphic bricks in $\modu A$ appear as the label of cover relation of such torsion classes; for more details, see \cite{AIR, BCZ, DIJ} and the references therein. 
The following remark supplements the assertion of Proposition \ref{Prop: min-image characterization of left-finite bricks}, with more information on $X_B=T_B \oplus \tau^{-1}F_B$. 
Recall that, for a module $M \in \modu A$, by $g_M$ we denote the $g$-vector of $M$. Moreover, if $I(S)$ is an indecomposable injective module, we write $\tau^{-1}F = P(S)[1]$, where the $g$-vector of $P(S)[1]$ is nothing but the negative of the $g$-vector of $P(S)$.

\begin{remark}\label{Rem: labeling bricks for distinct supports}
Using the notations from Lemma \ref{Lem: cases for labeling bricks} and Proposition \ref{Prop: min-image characterization of left-finite bricks}, note that if $\mathcal{T}_1$ and $\mathcal{T}_2$ have the same support, then both $T$ and $\tau^{-1}F$ are $\tau$-rigid, and they form an exchange pair; more explicitly, there is a $\tau$-rigid module $M$ such that both $M \oplus T$ and $M \oplus \tau^{-1}F$ are support $\tau$-tilting modules. In this case, $\mathcal{T}_1 = \Fac(M \oplus \tau^{-1}F)$ while $\mathcal{T}_2 = \Fac(M \oplus T)$. 
If the supports of $\mathcal{T}_1$ and $\mathcal{T}_2$ are different, by the proof of Lemma \ref{Lem: cases for labeling bricks}, they differ exactly by a simple module $S$. Nevertheless, the situation is similar to the exchange phenomenon discussed in the previous case, but one needs the language of support $\tau$-tilting pairs; here, the support $\tau$-tilting pair for $\mathcal{T}_1$ is of the form $(M,P \oplus P(S))$ and the support $\tau$-tilting pair for $\mathcal{T}_2$ is of the form $(M \oplus T,P)$. As before, we use the notation $P(S)[1] = \tau^{-1}F$. Therefore, in both cases, we can refer to $T$ and $\tau^{-1}F$ as the exchange pair of the cover relation $\mathcal{T}_1 \lessdot \mathcal{T}_2$.

We let $\sigma := g_T + g_{\tau^{-1}F}$, which is given by the formula
$$ \sigma(-)  =  {\rm dim}\Hom_A(T, -) - {\rm dim}\Hom_A(-, \tau T) + {\rm dim}\Hom_A(\tau^{-1}F, -) - {\rm dim}\Hom_A(-,F).$$
We note that when $F$ is not injective (equivalently, when $\mathcal{T}_1$ and $\mathcal{T}_2$ have the same support), the stability condition $\sigma$ is simply the $g$-vector $g_X$ of $X$.

We claim that $B$ is $\sigma$-stable. First observe that ${\rm dim}\Hom_A(T, B) = 1$, which follows from the description of $B$ as a quotient of $T$; see \cite{As1}. Dually, we have ${\rm dim}\Hom_A(B, F) = 1$. Hence, 
\begin{eqnarray*}   \sigma(B) & = & {\rm dim}\Hom_A(T, B) - {\rm dim}\Hom_A(B, \tau T) + {\rm dim}\Hom_A(\tau^{-1}F, B) - {\rm dim}\Hom_A(B,F)\\
        & = & {\rm dim}\Hom_A(T, B) - {\rm dim}\Hom_A(B,F)= 0.
\end{eqnarray*} 
Now, let $L$ be a proper non-zero submodule of $B$. Observe that $L \in \mathcal{F}_2$, because $B$ is torsion almost torsion-free with respect to the torsion pair $(\mathcal{T}_2, \mathcal{F}_2)$, hence $\Hom_A(T,L)=0$. Similarly, $\Hom_A(\tau^{-1}F,L)=0$, because $\tau^{-1}F \in \mathcal{T}_1 \subset \mathcal{T}_2$. Since we have a monomorphism $B \to F$, we get $\Hom_A(L,F) \ne 0$, thus $\sigma(L) < 0$, which shows that $B$ is $\sigma$-stable. If $F$ is not injective, the stability condition $\vartheta$ from Proposition \ref{prop:stableimage} constructed from the $g$-vector $\sigma$ of $X_B$ is just a scalar multiple of $\sigma$. Finally, note that $\sigma$ lies in the relative interior of the wall $\Theta_B$ of the brick $B$; see \cite[Lemma 2.14]{AI}.

\begin{remark}
We note that a similar construction of the stability condition for the bricks that label cover relations between functorially finite torison classes has been considered in the literature; see \cite{BST,Yu} and the references therein. 
Observe that the above construction does not use the geometric properties of the cone, and relies on the information of the torsion pairs of the cover relation. 
In fact, for all bricks $B$ in Proposition \ref{prop:stableimage}, hence almost all bricks that label a cover relation between functorially finite torsion classes (Proposition \ref{Prop: min-image characterization of left-finite bricks}), we give a canonical stability condition $\vartheta_B$ such that $B$ is $\vartheta_B$-stable: we consider the module $X_B$ such that $B$ is a minimal non-zero image of $X_B \to \tau X_B$, and we set $\vartheta_B = (\dim B)g_{X_B}-g_{X_B}(B)\ell$, where we recall $\ell = [A]$.
 \end{remark}
\end{remark}

\section{Proof of the main results and consequences}
\label{Sec:proofs}

We begin this section by showing an important result that we remarked on in Section \ref{Sec: Introduction}. First, recall that if $Z \in \Irr(A,\dd)$, then there exist finitely generated projective modules $P_0, P_1$ with no common nonzero direct summand, such that an element in general position in $Z$ has a minimal projective presentation of shape $P_1 \to P_0$; see Proposition \ref{prop:presentations}. That being the case, we let $\sigma:=[P_0]-[P_1]$ and use this notation in the following.

\begin{theorem}\label{Thm: from tau-regular comp. to brick comp}
Let $A$ be an algebra and $Z \in \Irr(A,\dd)$ be a $\tau$-regular component with $c(Z)>0$. 
There exists a dimension vector $\bb\leq\dd$, and a brick component $W\subseteq\rep(A,\bb)$ with $c(W)>0$, such that the general points $B$ of $W$ are stable for $ \vartheta=(\dim B)\sigma-\sigma(B)\ell$.
\end{theorem}

\begin{proof}
Let $Z\subseteq\rep(A,\dd)$ be a $\tau$-regular component
with $c(Z)>0$, and consider the total dimension $r=\sum_i d_i$. Moreover, consider the projective modules $P_1,P_0$ above and a nonempty open subset
$U\subseteq H=\Hom_A(P_1,P_0)$, as in
Proposition~\ref{prop:presentations}.

For every $p\in U$, we have $\Hom_A(X_p,\tau X_p)\ne0$. 
Choose a nonzero map $f_p: X_p\to\tau X_p$ whose image has minimal dimension, and put $B_p=\im f_p$.
By Lemma~\ref{lem:brickimage}, $B_p$ is a brick of dimension at most $r$.
Moreover, the image factorization gives $X_p\twoheadrightarrow B_p\hookrightarrow\tau X_p$, and the second map is nonzero, so $p\in{\calD}_{B_p}$, where we recall that $$\calD_{B_p}=\{s\in H\ \,|\,
 \text{there is an epi }X_s\twoheadrightarrow B_p,\;\; \Hom(B_p,\tau X_s)\ne 0\}.$$

We first show that $T:=\{B_p \,|\, p\in U\}$ is an infinite set. For the sake of contradiction, assume otherwise and let $T=\{B_1,\ldots,B_t\}$, for a positive integer $t$.
The definition of $\calD_B$ depends only on the isomorphism class of $B$. Therefore, Lemma~\ref{lem:incidence} gives
\begin{equation*}
 U\subseteq\bigcup_{i=1}^t\calD_{B_i}
 \subseteq\bigcup_{i=1}^t F_{B_i},
\end{equation*}
where each $F_{B_i}\subsetneq H$ is a proper closed set.
This implies that the non-empty open set $U$ is contained in a proper closed set, which is the desired contradiction, because $H$ is irreducible.

Because $B_p$ is a quotient of $X_p$, every  $\underline{\dim}B_p$ is componentwise bounded by $\dd$, and since $T=\{B_p \,|\, p\in U\}$ is an infinite set, some $\bb\leq \dd$ is attained for infinitely many of $\underline{\dim}B_p$.

For $\sigma=[P_0]-[P_1]$ above, let
$\vartheta=(\dim B)\sigma-\sigma(B)\ell$.
By Proposition~\ref{prop:stableimage},
every brick $B_p$ of dimension vector $\bb$ is stable for this weight. Since $\rep(A,\bb)$ has only finitely many
irreducible components, one of them, say $W$, contains infinitely many of these isomorphism classes. Its $\vartheta$-stable locus is
nonempty and open, hence dense.
\end{proof}

\begin{remark}
Let $\sigma \in K_0(\proj A)$ be a non-rigid $g$-vector (namely, $\sigma$ lies outside of the $\tau$-tilting fan). Notably, the $g$-vector $\vartheta$ generated in Theorem \ref{Thm: from tau-regular comp. to brick comp} is also non-rigid; because there are infinitely many non-isomorphic $\vartheta$-stable bricks. 
In fact, $\sigma$ and $\vartheta$ lie on different rays if and only if $\sigma({\mathbf b}) \ne 0$. That being the case, we can iterate the process; starting from $\vartheta$, we may return to $\sigma$, or obtain a new ray outside the $\tau$-tilting fan. Note that when $A$ is $E$-tame (which properly includes all tame algebras; see \cite{PY}), then $\sigma$ satisfies $\sigma({\mathbf b})=0$, hence $\vartheta$ and $\sigma$ lie on the same ray; for details, see \cite[Remark 8.5]{MP1}. 
\end{remark}

\begin{proof}[Proof of Theorem A]
Provided that $A$ is $E$-infinite, the definition yields a $\tau$-regular component $Z$ with $c(Z)>0$. Then, from Theorem \ref{Thm: from tau-regular comp. to brick comp} we get a brick component $W$ which yields the required stable family and shows that $A$ is brick-continuous.
\end{proof}

We are now ready to prove Theorem B.

\begin{proof}[Proof of Theorem B]
Assume that $A$ is not $E$-finite and let $\theta$ denote a non-rigid $g$-vector. It follows from Theorem A that one can find a brick component $W$ and an integral weight $\vartheta$, which depends on $\theta$, such that $W$ has infinitely many $G_{\bf b}$-orbits of $\vartheta$-stable points, where ${\bf b}$ denotes the dimension vector of the bricks in $W$. In particular, $\mathcal{M}(A,W)^{\vartheta-ss}$ is positive dimensional. Thus, $A$ is not stably-discrete.

Conversely, assume that $A$ is not stably-discrete. Then there exist a dimension vector ${\bf d}$, an irreducible component $Z$ in $\Irr(A,{\bf d})$ and an integral weight $\vartheta$ such that $\mathcal{M}(A,Z)^{\vartheta-ss}$ is positive dimensional. The points in $\mathcal{M}(A,Z)^{\vartheta-ss}$ are in bijection with the $\vartheta$-polystable points. Since $\mathcal{M}(A,Z)^{\vartheta-ss}$ is positive dimensional, it implies that there exists ${\bf b} \le {\bf d}$ and $W \in \Irr(A,{\bf b})$ such that $\mathcal{M}(A,W)^{\vartheta-ss}$ is positive dimensional and contains infinitely many $\vartheta$-stable points. In particular, $W$ is a brick component with $c(W)>0$. 
We claim that $\vartheta$ is a non-rigid $g$-vector, which would imply that $A$ is $E$-infinite. If $\vartheta$ is rigid, then it is known (see \cite{As2}) that the numerical torsion class $\overline{\mathcal{T}}_\vartheta$ associated to $\vartheta$ is functorially finite. It follows from \cite{As1} that the wide subcategory $\mathcal{W}_\vartheta$ of $\vartheta$-semistable modules is left-finite. In particular, it has at most $({\rk A})$-many non-isomorphic $\vartheta$-stable modules, which is the desired contradiction.
\end{proof}

We observe that Corollaries C, D, E, and F follow from the observations and remarks that we made around each one in Section \ref{Sec: Introduction}, and the references therein. Hence, we omit their proofs.

\section{Example}
To complement our observations on the reduction to the 2nd bBT conjecture from Section~\ref{Sec: Introduction}, with regard to Corollary~F and Remark following that, here we present an example of an algebra $A$ that is minimal brick-infinite and $g$-tame, but we observe that $A$ does not admit the $\tau$-convergence property. In particular, this provides a non-example to a question posed in \cite[Subsection 6.2]{MP1}. The following algebra has also appeared in \cite{Wa}, in the study of minimal brick-infinite algebras of rank $2$. 
More precisely, the $\tau$-tilting fan of $A$ (also called the $g$-vector fan) is only missing the ray of $(2,-1)$, and the quadratic form on the $g$-vectors, given by the Cartan matrix of $A$, does not vanish at $(2,-1)$; for details, see \cite{MP1}. 
Although $A$ satisfies conditions (i) and (ii) of Corollary~F, it does not have the $\tau$-convergence property. We note, however, that $A$ admits infinitely many unfaithful bricks, namely, it does not satisfy condition (iii) of Corollary F.

\begin{example}\label{Ex: non-tau-convergence}
Take the algebra $A=kQ/I$, given by the following quiver $Q$
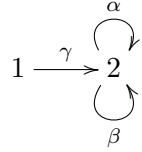
$$\xymatrix{1 \ar[r]^\gamma & 2  \ar@(ul,ur)^\alpha  \ar@(dl,dr)_\beta} $$
and the admissible ideal $I:=\langle \alpha^2, \alpha\beta, \beta\alpha, \beta^2\rangle$, generated by quadratic relations. It follows from \cite{Wa} that $A$ is minimal brick-infinite, and a simple argument shows that $A$ is a non-distributive wild algebra; see \cite{Mo1,Mo2} and references therein. Moreover, infinitely many bricks of $A$ of dimension vector $(1,2)$ are unfaithful. To check $g$-tameness of $A$, it is sufficient to consider the $g$-vectors of shape $(a,-b)$ with $a, b$ positive, and prove that these are all rigid, except when $a = 2b$. 

If $a \le b$, then a generic cokernel $P_2^b \to P_1^a$ is a direct sum of copies of $S_1$. Therefore, such $g$-vectors are rigid. 
Assume now that $a > b$. Let $R := e_2Ae_2$, which is a $3$-dimensional local algebra. We can think of a morphism $P_2^b \to P_1^a$ as being given by an $a \times b$ matrix with entries in $\gamma R$. Recall that ${\rm Aut}(P_2^b) \times {\rm Aut}(P_1^a)$ acts on the presentation space $\Hom(P_2^b,P_1^a)$ by the rule $(U,V)\cdot h = VhU^{-1}$. By acting via ${\rm Aut}(P_2^b)$, if $f$ is in general position in $\Hom(P_2^b,P_1^a)$, then $f$ can be reduced to matrix form $g:=[\gamma I_b, M]^T$ where $M$ is an $(a-b)\times b$ matrix with entries in $\gamma R$. By further acting via ${\rm GL}_a \subset {\rm Aut}(P_1^a)$, we can further assume that $M$ has entries in $\gamma {\rm rad}R$. In other words, entries of $M$ become linear combinations of $\gamma\alpha, \gamma\beta$. Now, we consider the map $$\psi_g: {\rm End}(P_2^b) \times {\rm End}(P_1^a) \to \Hom(P_2^b, P_1^a)$$ given by $(U,V)\mapsto gU-Vg$. It follows from \cite[Proposition 3]{EJR} that the cokernel of this map is zero precisely when $g$ represents a $2$-term pre-silting complex. Hence, our task is to prove that if $a \ne 2b$, this map is surjective.
If we start with an arbitrary element in $\Hom(P_2^b,P_1^a)$ represented by the block matrix $[X, Y]^T$, where $X$ is $b \times b$ and $Y$ is $(a-b)\times b$, then write $Y=Y_0+Y_1$ where $Y_1$ has entries in $\gamma{\rm rad}R$ and $Y_0$ has constant entries. We can simply use
$$(U,V) = \left(X', \left(\begin{array}{cc}
     0 & 0 \\
     -Y_0& 0\\ 
\end{array}\right)\right), \quad X=\gamma X'$$
which has image $[X, Y_0+MX]^T$ under $\psi_g$, where we remark that $MX$ has entries in $\gamma{\rm rad}R$. Hence, in order to study the image of $\psi_g$, it suffices to consider an element of the form $[0, N]^T$ where $N$ has entries in $\gamma{\rm rad}R$. Writing $V$ using blocks, we get that $[0, N]^T$ lies in the image of $\psi_g$ if and only if there exists $(U,V)$ with
\begin{equation}\label{eqn:block matrix}
    [U-V_{11}-V_{12}M, MU - V_{21}-V_{22}M] = [0, N].
\end{equation}
Writing $U$ as $U=U_0+U_1$ where $U_1$ has entries in $\gamma{\rm rad}R$ and $U_0$ has constant entries, Equation \eqref{eqn:block matrix} yields $U_0 = V_{11}, V_{12}M = U_1, V_{21}=0$ and $MU_0-V_{22}M=N$. Here, we have used that the entries of $M$ are in $\gamma{\rm rad}R$, so $MU_1=0$. Hence, solving Equation \eqref{eqn:block matrix} amounts to deciding whether there exist $V_{11}, V_{22}$ such that $MV_{11}-V_{22}M=N$. Splitting the matrices $M,N$ into the $\gamma\alpha$ and $\gamma\beta$ parts, we get $M = M_{\gamma\alpha} + M_{\gamma\beta}$ and $N = N_{\gamma\alpha}+N_{\gamma\beta}$. The problem can therefore be reduced to surjectivity of
$${\rm Mat}_b(k) \times {\rm Mat}_{a-b}(k) \to {\rm Mat}_{(a-b)\times b}(k) \times {\rm Mat}_{(a-b)\times b}(k)$$ with the rule $(V_{11}, V_{22}) \mapsto (M_{\gamma\alpha}V_{11}-V_{22}M_{\gamma\alpha}, M_{\gamma\beta}V_{11}-V_{22}M_{\gamma\beta})$.
The latter is precisely the analogue of the map $\psi_g$ above, but for the $2$-Kronecker quiver and $g$-vector $(a-b,-b)$. Since $M$ is in general position from the choice of $g$, we know that this map is surjective when $a-b\ne b$, that is, when $a \ne 2b$. We have therefore proved that the $\tau$-tilting fan of $A$ can only miss one ray of $\mathbb{R}^2$, namely the ray of $(2,-1)$. Hence, $A$ is $g$-tame. 

Now, it is easy to check that if $C_A$ is the Cartan matrix of $A$, then $(2,-1)C_A(2,-1)^T=1 \ne 0$, thus $A$ cannot have the $\tau$-convergence property, because $\tau$-converge would imply that $(2,-1)$ is a zero of the Cartan quadratic form.
\end{example}

\vskip 1cm

\noindent\textbf{Acknowledgments.}
KM was supported by Early-Career Scientist JSPS Kakenhi grant number 24K16908, and also by Fields Institute Elliott-Yui Distinguished Visitor Fellowship. CP was supported by the Natural Sciences and Engineering Research Council of Canada (RGPIN-2026-05988) and by the Canadian Defence Academy Research Programme. The authors thank Sota Asai and Osamu Iyama for some stimulating discussions related to the study of the brick-Brauer-Thrall Conjectures. We also thank Sota Asai and Haruhisa Enomoto for some helpful comments on the first version of this manuscript.

\medskip

\noindent {\bf Declaration of AI use}. Artificial intelligence (GPT-$6$ Astra) has been used to fill in some arguments in the proof of Lemma \ref{lem:incidence}, to find the correct shape of the stability condition in Proposition \ref{prop:stableimage}, and to find Example \ref{Ex: non-tau-convergence}. The authors are responsible for the content.

\end{document}